\documentclass{article}
\usepackage[margin=1in]{geometry}

\usepackage{amsmath}
\usepackage{amssymb}
\usepackage{url}

\usepackage{todonotes}
\usepackage{xspace}
\usepackage{amsthm}
\usepackage{hyperref}
\usepackage{cleveref}
\usepackage{algorithm}
\usepackage{caption}
\usepackage[noend]{algpseudocode}
\usepackage{cancel}
\usepackage{diagbox}
\usepackage{multirow}

\newtheorem{theo}{Theorem}[section]
\newtheorem{defn}[theo]{Definition}
\newtheorem{rem}[theo]{Remark}
\newtheorem{exmp}[theo]{Example}
\newtheorem{prop}[theo]{Proposition}

\newtheorem{cor}[theo]{Corollary}
\newtheorem{con}[theo]{Conjecture}

\newcommand{\Sat}{\operatorname{Sat}}
\newcommand{\Act}{\operatorname{Act}}
\newcommand{\Pos}{\operatorname{Pos}}

\newcommand{\supp}{\operatorname{supp}}
\newcommand{\R}{\mathbb{R}}
\newcommand{\Rmax}{\R_{\max}}
\newcommand{\Rmin}{\R_{\min}}
\newcommand{\bomega}{\bigoplus_{\omega}}
\newcommand{\idxset}{\mathcal{I}}
\newcommand{\solset}{\mathcal{S}}
\newcommand{\dbm}[1]{\mathtt{DBM}(#1)}
\newcommand{\zone}[1]{\mathtt{Zone}(#1)}

\newcommand{\tred}[1]{\textcolor{red}{#1}}
\newcommand{\ceils}[1]{\lceil#1\rceil}

\makeatletter
\title{On the maxmin-$\omega$ eigenspaces and their over-approximation by zones}

\author{Muhammad Syifa'ul Mufid\footnote{Department of Mathematics, Institut Teknologi Sepuluh Nopember, Jl Arif Rahman Hakim, Surabaya,
            60111, Indonesia. Email: syifaul.mufid@matematika.its.ac.id}\and 
            Ebrahim Patel\footnote{School of Computing and Mathematical Sciences, University of Greenwich, Old Royal Naval College, Park Row, London, SE10 9LS, UK. Email: e.patel@greenwich.ac.uk}\and
            Serge\u{\i} Sergeev\footnote{University of Birmingham, School of Mathematics, Edgbaston B15 2TT, UK. Email: s.sergeev@bham.ac.uk}}

\date{}

\begin{document}
\maketitle
\makeatother

\begin{abstract}
Maxmin-$\omega$ dynamical systems were previously introduced as a generalization of dynamical systems expressed by tropical linear algebra. To describe steady states of such systems one has to study an eigenproblem of the form $A\otimes_{\omega} x=\lambda+x$ where $\otimes_{\omega}$ is the maxmin-$\omega$ matrix-vector multiplication. This eigenproblem can be viewed in more general framework of nonlinear Perron-Frobenius theory. However, instead of studying such eigenspaces directly we develop a different approach: over-approximation by zones. These are traditionally convex sets of special kind which proved to be highly useful in computer science and also relevant in tropical convexity. We first construct a sequence of zones over-approximating a maxmin-$\omega$ eigenspace. Next, the limit of this sequence is refined in a heuristic procedure, which yields a refined zone and also the eigenvalue $\lambda$ with a high success rate. Based on the numerical experiments, in successful cases there is a column of the difference bound matrix (DBM) representation of the refined zone which yields an eigenvector. 

\end{abstract}

\section{Introduction}
\label{s:intro}
Let us consider a typical discrete event dynamical system modelled by means of the tropical (max-plus) linear algebra. In this system, there is a team of $n$ machines $P_1,\ldots, P_n$ working in cycles (it could also be processors or individuals in a different application). The machines crucially depend on each other as each machine needs other machines' products to proceed (alternatively, we might imagine the propagation of pieces of information in a network).  The starting times of these machines in the next cycle depend on their starting times in the current cycle. The time lag necessary to bring a product from machine $i$ to machine $j$ is denoted by $A(i,j)$. Suppose that $x_i(k)$ is the starting time of the $i$th machine in the $k$th cycle and that each machine needs to wait for 
all products made in the previous production cycle to arrive before they can start the present production cycle. In this case the starting times in the $(k+1)$th 
production cycle are given by 
\begin{equation}
\label{MaxDynamics}
x_i(k+1)=\max_{j=1,\ldots,n} A(i,j)+x_j(k)\quad\forall i=1,\ldots,n 
\end{equation}
Now suppose that each machine can proceed after the first product produced in the previous cycle arrives. Note that this is much easier to imagine in a network application where the machines are replaced with individuals, who propagate some pieces of information~\cite{EPNewThreshold}. In this case we have 
\begin{equation}
\label{MinDynamics}
x_i(k+1)=\min_{j=1,\ldots,n} A(i,j)+x_j(k)\quad\forall i=1,\ldots,n 
\end{equation}
In the present paper we consider an intermediate case where each machine waits only for $(\omega\cdot 100)\%$ of products from all machines to arrive, and then starts a new round of production immediately. Here, $0\leq\omega\leq 1$, and we will only consider $\omega=p/n$ where $p=1,\ldots,n$. In terms of $p$, it is the $p$th fastest product received by machine $i$ from machine $j$ which determines the next starting time $x_i(k+1)=A(i,j)+x_j(k)$ of the $i$th machine. We can express this mathematically as $x(k+1)=A\otimes_{\omega} x(k)$, where $(A\otimes_{\omega} x(k))_i$ is the $p$th smallest number among $A(i,1)+x_1(k),\ldots,A(i,n)+x_n(k)$. 

For starting time vectors $x(k)$ to be well predictable, we would like the initial vector $x=x(0)$ to satisfy $A\otimes_{\omega} x=\lambda_p +x$, where $p=\omega n$. That is, we would like it to be a maxmin-$\omega$ eigenvector. Together with an observation (after some experimentation) that $x(k)$ converges to such an eigenvector after a finite time at least in some cases~\cite{EPthesis}, this provides a motivation to study the sets of such eigenvectors, which we call the maxmin-$\omega$ eigenspaces. 

Min-plus and max-plus eigenspaces, i.e., sets of vectors satisfying $A\otimes_{\omega} x=\lambda_p+x$ for $\omega=1/n$ and $\omega=1$ respectively, have been well studied \cite{baccelli92,Butkovic2010} and comprehensively described, but this is not the case for intermediate $\omega\colon 1/n<\omega<1$. For an initial approach one can observe that the maxmin-$\omega$ matrix-vector multiplication is a special case of the so-called min-max functions studied in \cite{GG-04,Gun1994, Subiono} (among many other papers) and, more generally, increasing and additively homogeneous functions. This observation is very useful since it allows us to apply the nonlinear Perron-Frobenius theory and, in particular, the monotonicity of spectral radius~\cite{Nuss86}. Note that another relevant concept is that of ambitropical cones \cite{AGV}: the maxmin-$\omega$ eigenspaces studied in the present paper can be seen as a special case of them. In particular, \cite{AGV} Proposition 4.18 implies that the maxmin-$\omega$ eigenspaces are connected sets, unlike the solution sets of maxmin-$\omega$ linear systems $A\otimes_{\omega} x=b$ which we studied in \cite{Maxminomega1}. However, these eigenspaces are not convex in a tropical or traditional sense and, although they can be described by generators (see \cite{AGV} Section 8), such a description is hard to be found algorithmically.

For these reasons, we take a different approach by developing the idea of over-approximating the maxmin-$\omega$ eigenspaces by zones. These are convex sets of special type, highly useful in some areas of computer science~\cite{Dill,Mine} and important in tropical convexity~\cite{DS-04}. Zones are also closely related to alcoved polytopes studied in \cite{LP-07,MdlP} (among many other works). The first main result of the present paper is that we construct a sequence of zones over-approximating a maxmin-$\omega$ eigenspace (see \Cref{a:sequence}). We then also suggest a heuristic procedure based on a bisection principle that refines the limit of that sequence and, with a good success rate, also yields the corresponding eigenvalue $\lambda_p$. 

The rest of this paper is organized as follows. \Cref{s:prel} is devoted to the necessary preliminaries on tropical linear algebra, nonlinear Perron-Frobenius theory, zones and their difference bound matrix (DBM) representations. We also introduce saturation graphs for the maxmin-$\omega$ eigenproblem, which are similar to the saturation graphs for the max-plus eigenproblem (\cite{Heid2006}, Section 4.2). \Cref{s:main} presents the main results of the paper. The sequence of over-approximations by zones is constructed in \Cref{ss:sequence}. Next, some sufficient conditions for entries $(i,j)$ to be active or inactive are given in \Cref{ss:active}. Finally, \Cref{ss:refinement} presents a heuristic procedure to refine the limit of the sequence of zones based on the observations of the previous subsection and on the bisection principle.

\section{Preliminaries}
\label{s:prel}


\subsection{Max-plus, min-plus and maxmin-$\omega$ algebras}
\label{s:algebra}

We start with some basic definitions of the max-plus semiring and linear algebra over that semiring.

\begin{defn}[Max-plus semiring \cite{baccelli92,Butkovic2010, Heid2006}] 
\label{def:max-plus}
{\rm The max-plus semiring, often also called the max-plus algebra, is the set $\Rmax$ equipped with arithmetical operations of ``addition'' $\oplus$ defined by $a\oplus b := \max\{a,b\}$
and ``multiplication'' $a\otimes b:=a + b$.}
\end{defn}
\begin{defn}[Max-plus matrices \cite{baccelli92,Butkovic2010,Heid2006}]
{\rm The set of $m\times n$ matrices over the max-plus algebra is denoted by $\Rmax^{m\times n}$. For $n=1$, the set of column vectors with $m$ elements in the max-plus algebra is denoted by $\Rmax^m$ instead. Given a matrix $A$, we define $A(i,j)$ as the element of $A$ at the $i$-th row and $j $-th column. 
As in the classical linear algebra, the operations in \Cref{def:max-plus} can be extended to matrices and vectors: for $ A,B\in \Rmax^{m\times n}, C\in \Rmax^{n\times p}$ and $\alpha\in\Rmax$, 
\begin{align*}
\nonumber [A\oplus B](i,j) &= A(i,j)\oplus B(i,j)=\max\{A(i,j),B(i,j)\},\\
\nonumber [A\otimes C](i,j) &= \bigoplus_{k=1}^n A(i,k)\otimes C(k,j)=\max_{1\leq k\leq n}\{A(i,k)+C(k,j)\},\\
\nonumber [\alpha\otimes A](i,j) &= \alpha\otimes A(i,j)=\alpha+A(i,j).
\end{align*}
For a square matrix $A\in \Rmax^{n\times n}$, the $k$-th power of $A$ is defined as $A^{k}=A\otimes A\otimes \ldots \otimes A$ ($k$ times). For $k=0$, the $0$-th power of $A$ yields an $n\times n$ identity matrix in the max-plus algebraic sense. All diagonal entries of this matrix are $0$ and all off-diagonal elements are $-\infty$.}
\end{defn}

The following defines the natural dual of the max-plus semiring, which is called the min-plus semiring.

\begin{defn}[Min-plus semiring]
{\rm The min-plus semiring (also called the min-plus algebra) is the set $\Rmin$ equipped with ``addition'' $\oplus^\prime$ defined by $a\oplus^\prime b=\min\{a,b\}$ and ``multiplication'' $\otimes^\prime$ defined by $a\otimes^\prime b=a+b$.}
\end{defn}

We will omit the formal definition of notations and operations with matrices in min-plus algebra as they are defined similarly to those of max-plus algebra. Observe that $a\oplus^\prime b=-((-a)\oplus (-b))$, under the convention that $-(+\infty)=-\infty\in \Rmax$ and $-(-\infty)=+\infty\in \Rmin$.

\begin{defn}[Maxmin-$\omega$ operation \cite{EPthesis}]
{\rm Given a multiset (i.e., a set with possible repetitions of elements) $S\subseteq \mathbb{R}$ with $n$ elements, the \textit{maxmin-$\omega$} operation 
\begin{equation}
\label{eq:omega-operation}
    \bomega S
\end{equation}
yields the $\lceil \omega n \rceil$-th smallest element of $S$ for $0<\omega \leq 1$. For the sake of simplicity, without loss of generality, we always assume that $\omega=p/n$ for some $p\in [n]$ where $[n]$ conventionally denotes the set $ \{1,\ldots,n\}$.}
\end{defn}

It is straightforward to see that when $\omega = 1/n$ (resp. $\omega=1)$, operation \eqref{eq:omega-operation} corresponds to min-plus addition (resp. max-plus addition).

The following notation helps to express 
the maxmin-$\omega$ operation as a rather elementary function involving both $\oplus$ and $\oplus^\prime$ arithmetics.

\begin{defn}[The set of $k$-element subsets]
{\rm For each set $S$, notation $|S|$ refers to the cardinality of $S$. For $0\leq k\leq |S|$, we define $\mathcal{P}(S,k)$ as the set of all $k$-element subsets of $S$.} 
\end{defn}

 We now express \eqref{eq:omega-operation} as the combination of $\oplus$ and $\oplus^\prime$ operations.

\begin{prop}[\cite{EPthesis}]
\label{prop:max-min-function}
Given a multiset $S=\{s_1,\ldots,s_n\} \subset \R $ and suppose that $\omega =p/n$ for $p\in [n]$, then
\begin{align}
\label{eq:bomega-min}
    \bomega S &= \min_{C\in \mathcal{P}([n],~p)} \max_{i\in C} s_i
    \\
    \label{eq:bomega-max}
    &= \max_{C\in \mathcal{P}([n],~n+1-p)} \min_{i\in C} s_i
\end{align}
\end{prop}
\begin{proof}
Omitted.
\end{proof}

\begin{defn}[CNF and DNF]
{\rm Expressions \eqref{eq:bomega-min} and \eqref{eq:bomega-max} are, respectively, called the conjunctive normal form (CNF) and the disjunctive normal form (DNF) of the maxmin-$\omega$ operation.}
\end{defn}

\begin{rem}
{\rm In this paper, following the approach of~\cite{Maxminomega1}, we limit the scope to maxmin-$\omega$ operations with finite numbers, since including both $-\infty$ and $+\infty$ in this operation can be confusing (for example, consider the result of $-\infty+\infty$) and not practical in view of the applications that can be considered~\cite{EPthesis}.
}
\end{rem}

\subsection{Elements of Perron-Frobenius theory for the maxmin-$\omega$ eigenproblem}
\label{ss:PF}

We are going to consider the following generalization of the max-plus and min-plus eigenproblems.

\begin{defn}[Maxmin-$\omega$ eigenproblem]
{\rm Given a matrix $A\in \R^{n\times n}$ and $\omega=p/n$ for $p\in [n]$, we call $\lambda_p$ the maxmin-$\omega$ eigenvalue of $A$ if there exists $x\in \R^n$ such that
\begin{equation}
\label{eq:problem}
A\otimes_\omega x=\lambda_p\otimes x.
\end{equation}
Such $x$ is called a maxmin-$\omega$ eigenvector of $A$ associated with $\lambda_p$. The set of these eigenvectors is denoted by $E_p(A)$ and is called the maxmin-$\omega$ eigenspace of $A$ associated with $\lambda_p$.} 
\end{defn}

Algebraically, because of \Cref{prop:max-min-function}, the eigenproblem \eqref{eq:problem} can be expressed as a system of equations where maxmin-$\omega$ operations are in CNF:
\begin{align}
    \min_{C\in \mathcal{P}([n],~p)} \max_{j\in C} A(i,j)+x_j=\lambda_p+x_i~\text{for}~i\in [n],
\end{align}
or in DNF:
\begin{align}
    \max_{C\in \mathcal{P}([n],~n+1-p)}\min_{j\in C} A(i,j)+x_j
    =\lambda_p+x_i~\text{for}~i\in [n].
\end{align}

These forms~\eqref{eq:bomega-min} and~\eqref{eq:bomega-max} show that the maxmin-$\omega$ eigenproblem, while being a generalisation of max-plus and min-plus eigenproblems, is a special case of the eigenproblem for the so-called {\em min-max functions} studied, e.g., in \cite{GG-98,Gun1994,Subiono,SO-97}. Furthermore, any min-max function belongs to the class of functions $f\colon (\R\cup\{-\infty\})^n \to (\R\cup\{-\infty\})^n$ which satisfy the following properties: 
\begin{itemize}
    \item[1.] $f(\lambda+x)=\lambda+f(x)$ for $\lambda\in \R$ \hfill (additive homogeneity),
    \item[2.] $x\leq y\Rightarrow f(x)\leq f(y)$ \hfill (monotonicity).
\end{itemize}
For such functions we recall the following result, which extends the Perron-Frobenius theory for nonnegative matrices.
\begin{prop}[Coro. of Nussbaum~\cite{Nuss86} Theorem 3.1, also \cite{TropicalMPG} Lemma 2.8]
Let $f\colon  (\R\cup\{-\infty\})^n\to \R\cup\{-\infty\}^n$ be an additively homogenous and monotone function. 
\begin{itemize}
    \item[1.] There is a vector $x$ with $\supp(x)\neq \emptyset$ and a scalar $\rho\in \mathbb{R}$ such that $f(x)=\rho+x$.
    \item[2.] Let $r$ be the largest $\rho$ for which such vector $x$ exists. Then one has the following identities:
    \begin{equation}
    \label{e:CW}
       \begin{split}
          r&=\max\{\alpha\mid f(x)\geq \alpha+x\  \text{for some $x\in \Rmax[n]$ with $\supp(x)\neq \emptyset$}\}\\
           &=\inf\{\beta\mid f(x)\leq \beta+x\ \text{for some $x\in\R^n$}\} 
       \end{split} 
    \end{equation}
\end{itemize}
\end{prop}
The second of the equations in~\eqref{e:CW} is referred to as the Collatz-Wielandt identity (since it generalizes the well-known Collatz-Wielandt identity for the spectral radius of a nonnegative matrix). 

For $x\in \Rmax^n$, let us introduce, in the usual notation, $\supp(x)=\{i\mid x_i\neq -\infty\}$.
The following result can be also observed.


\begin{prop}[\cite{GS-Separation,Gun1994}]
\label{p:supports}
Let $x$ and $y$ be such that $f(x)=\lambda+x$, $f(y)=\mu+y$ and $\supp(x)\subseteq\supp(y)$ for an additively homogeneous and monotone $f$. Then $\lambda\leq\mu$.
\end{prop}

It follows, in particular, that for every fixed support $S\subseteq [n]$ we can have no more than one eigenvalue, for which there is an associated eigenvector $x$ with $\supp(x)=S$. Thus the number of eigenvalues of an additively homogeneous and monotone function is finite and bounded from above by $2^n-1$. If such a function has an eigenvector with full support (i.e., a finite eigenvector), then the corresponding eigenvalue is the largest among all eigenvalues of this function. 

Let us now specialize to the maxmin-$\omega$ operation and establish the following fact. Here, for a maxmin-$\omega$ operation $A\otimes_{\omega}$, we consider its natural extension to an additively homogeneous and monotone function acting on $\Rmax^n$.

\begin{prop}
\label{p:lambdap}
Let $A\in\R^{n\times n}$ and $\omega=p/n$ for $p\in [n]$.
Furthermore, let $f^{(p)}\colon\Rmax^n\to\Rmax^n$ be defined by
\begin{equation}
\label{e:fpdef}
f^{(p)}(x):=\min_{C\in \mathcal{P}([n],~p)} \max_{j\in C} A(i,j)+x_j.
\end{equation}    
Then:
\begin{itemize}
    \item[{\rm (i)}] $f^{(p)}$ can have only finite eigenvectors $x$ associated with $\omega$ and a finite eigenvalue $\lambda_p$,
    \item[{\rm (ii)}] $f^{(p)}$ has a unique finite eigenvalue $\lambda_p$ for each $\omega$.
\end{itemize}    
\end{prop}
\begin{proof} 
(i): Suppose that $x$ is an eigenvector of $f^{(p)}$ with $\supp(x)=M$. Let us consider two cases:\\
\textbf{Case 1:} $|M|\leq n-p$. In this case let $i\in M$. Then we obtain 
\begin{equation*}
\min_{C\in \mathcal{P}([n],~p)} \max_{j\in C} A(i,j)+x_j=-\infty,
\end{equation*}
since there exists $C\subseteq [n]\backslash M$ with $|C|=p$ for which $\max_{j\in C} A(i,j)+x_j=-\infty$. This contradicts with $x$ being an eigenvector with a finite eigenvalue.\\
\textbf{Case 2:} $|M|>n-p$. In this case let $i\notin M$. Then 
\begin{equation*}
\min_{C\in \mathcal{P}([n],~p)} \max_{j\in C} A(i,j)+x_j>-\infty,
\end{equation*}
since $\max_{j\in C} A(i,j)+x_j>-\infty$ for all $C$ with $|C|=p$. This contradicts with $x$ being an eigenvector.

(ii) To show that $f^{(p)}$ has a finite eigenvalue associated with any $\omega$, it can be observed that $A\otimes_{\omega} e\geq \min_{i,j}A(i,j)+e$ where $e$ is an $n$-component vector with all entries equal to $0$, after which the first property of~\eqref{e:CW} can be applied. Denote this eigenvalue by $\lambda_p$, then the corresponding eigenvector has a full support by part (i). Proposition~\ref{p:supports} then implies that such an eigenvalue is unique. 
\end{proof}

\begin{cor}
Let $A\in\R^{n\times n}$ and $\omega=p/n$ for $p\in[n]$. Then $A$ has a unique eigenvalue with respect to the threshold $\omega$.
\end{cor}

\begin{prop}
    For any matrix $A\in \R^{n\times n}$, we have $\lambda_1\leq\lambda_2\leq \cdots\leq \lambda_n$.
\end{prop}
\begin{proof}
Consider $f^{(p)}$ defined by \Cref{e:fpdef} for $p\in [n]$. Since we have $f^{(p)}(x)\geq f^{(q)}(x)$ for $p\geq q$ and any $x\in \R[n]$, 
we can apply the second property of~\eqref{e:CW} to establish the result.
\end{proof}

\subsection{Saturation graphs}
\label{ss:Sat}

Inspired by the approach to the max-plus spectral theory in \cite{Heid2006}, we can directly generalize the notion of saturation graphs to the case of maxmin-$\omega$ operation.

\begin{defn}
   {\rm Let a matrix $A\in \R^{n\times n}$, a vector $x\in \R^n$ and a threshold $\omega=p/n$ for $p\in [n]$ be given. A saturation graph with respect to $A,\omega$ and $x$, denoted by $\Sat(A,\omega,x)$, is a directed graph whose node set is $[n]$ and  the arcs of which are $(i,j)$ such that $A\otimes_{\omega} x=A(i,j)+x_j$.}
\end{defn}


\begin{prop}
\label{p:satcycle}
Let $A\in\R^{n\times n}$.
\begin{itemize}
\item[(i)] Any eigenvalue $\lambda_p$ for $p=1,\ldots,n$ is a cycle mean of $A$.
\item[(ii)] If $x$ is an eigenvector associated with $\omega$ and $\lambda_p$, then $\Sat(A,\omega,x)$ contains a cycle whose mean weight is equal to $\lambda_p$.
\end{itemize}
\end{prop}
\begin{proof}
Let $x$ be an eigenvector of $A$ associated with $\omega$ and a finite eigenvalue $\lambda_p$.
In $\Sat(A,\omega,x)$ every node has an outgoing arc. If we start from a node $i$ and consider a path in $\Sat(A,\omega,x)$ from this node, then it ends up with a cycle $(i_1,i_2,\ldots,i_k)$. For the arcs of this cycle we have:
\begin{equation*}
A(i_1,i_2)+x_{i_2}=\lambda+x_{i_1}, A(i_2,i_3)+x_{i_3}=\lambda+x_{i_2},\ldots, A(i_k,i_1)+x_{i_1}=\lambda+x_{i_k}.     
\end{equation*}
Adding up these identities and cancelling $x_{i_1}+\ldots+x_{i_k}$ from the result we obtain that $\lambda$ is the cycle mean of $(i_1,\ldots,i_k)$. The cycle $(i_1,\ldots,i_k)$ belongs to $\Sat(A,\omega,x)$.
\end{proof}

\begin{exmp}
{\rm Consider
\begin{equation*}
A=
\begin{bmatrix}
4 & 7 & 2\\
5 & 2 & 5\\
6 & 3 & 1
\end{bmatrix}
\end{equation*}
and take $\lambda=4$ and $\omega=1/2$, so that $p=2$ (considering the second smallest term in $\{A(i,1)+x_1,A(i,2)+x_2,A(i,3)+x_3\}$ for each $i\in [3]$). There are 3 cycles with this cycle mean: $(1)$, $(1,3)$ and $(2,3)$. 

Let us first analyse the second and the third equations of $A\otimes_{\omega} x=\lambda+x$. In the second equation, either $4+x_2=5+x_1$ or $4+x_2=5+x_3$, and in the third equation, either $4+x_3=6+x_1$ or $4+x_3=3+x_2$. 
Let us first assume that $4+x_2=5+x_1$, $4+x_3=6+x_1$ and $x_3=0$ (as we can add any scalar to an eigenvector). Then, $x_1=-2$ and $x_2=-1$. In this case, however, $6+x_1$ is not the second minimal term among $(6+x_1,\; 3+x_2,\; 1+x_2)$, so this case is not possible. 
Now we assume that $4+x_2=5+x_1$, $4+x_3=3+x_2$ and $x_3=0$. Then, $x_2=1$ and $x_1=0$. This implies that also $4+x_2=5+x_3$ in the second equation, meaning that $(2,3)\in\Sat(A,1/2,x)$.\\
Let us now assume $4+x_2=5+x_3$, $4+x_3=6+x_1$ and $x_3=0$. Then, $x_1=-2$ and $x_2=1$. This implies that also $4+x_3=3+x_2$ in the third equation, again implying that 
$(2,3)\in\Sat(A,1/2,x)$. 

We thus obtain that any $1/2$-eigenvector associated with $4$ should have 
$(2,3)\in\Sat(A,1/2,x)$. Let us describe these eigenvectors. Setting $x_3=0$, we immediately obtain $x_2=1$ from the second and third equations. We also have from the same equations that $5+x_1\geq 5+x_3$ and $6+x_1\geq 3+x_2$ (note that the terms $2+x_2$ and $1+x_3$ have to be the smallest in the second and the third equation, respectively). This implies $x\geq 0$. We now turn to the first equation of the eigenproblem. There we have the terms $4+x_1$, $8$ and $2$, of which $4+x_1$ has to be the second smallest, which is true if and only if $2\leq 4+x_1\leq 8$. As already $x_1\geq 0$, we obtain $0\leq x_1\leq 4$.  The eigenspace therefore consists of the vectors
$$
\lambda\otimes
\begin{bmatrix}
[0,\ \  4]\\
1\\
0
\end{bmatrix}
$$
It is also easy to observe that $(1)\in\Sat(A,1/2,x)$ for any of these eigenvectors, but $(1,3)\notin\Sat(A,1/2,x)$ for any of these eigenvectors. 
From this we can conclude the following.} 
\end{exmp}
\begin{rem}
{\rm If $\lambda_{p}$ is a maxmin-$\omega$ eigenvalue of $A$ and $(i_1\ldots,i_k)$ is a cycle with mean weight $\lambda_{p}$, it is in general not true that there exists a maxmin-$\omega$ eigenvector $x$ associated with $\lambda_p$, for which  $(i_1\ldots,i_k)\in\Sat(A,p/n,x)$. This is quite different from the max-plus (respectively, min-plus) algebra in which the saturation graph $\Sat(A,1/n,x)$ (respectively, $\Sat(A,1,x)$) contains every cycle with the maximal (respectively, minimal) cycle mean for any finite eigenvector $x$.}
\end{rem}

\subsection{Zones and DBM representations}
\label{ss:zones}

The main idea of this paper is to over-approximate the sets of $\omega-$eigenvectors $E_p$ by zones, which are relatively simple but expressive and specific convex sets (here meaning convex in the usual sense).

\begin{defn}[\cite{Dill,Mine}]
    A zone $Z\subseteq \R^n$ is the set of all points in $\R^n$ that satisfy the constraints $x_i-x_j\sim_{ij} d_{ij}$ where $\sim_{ij} \hspace*{1ex}\in \{>,\geq\}$ and $d_{ij}\in \R\cup\{-\infty\}$ for $i,j\in [n]$. 
\end{defn}
\begin{defn}[\cite{Mufid2018}]
    Given a zone $Z$ in $\R^n$, the difference-bound matrices (DBM) representation of $Z$ is the pair of $n\times n$ matrices $(D,S)$ where $D(i,j)=d_{ij}\in\R\cup\{-\infty\}$ and
    \[
    S(i,j)=\left\{
    \begin{array}{cc}
        1, &~\text{if $\sim_{ij}\hspace*{0.75ex}=\hspace*{0.75ex}\geq$},  \\
        0, &~\text{if $\sim_{ij}\hspace*{0.75ex}=\hspace*{0.75ex}>$}.
    \end{array}
    \right.
    \]
such that $Z$ is the set satisfying all constraints 
$x_i-x_j\sim_{ij} d_{ij}$.
\end{defn}
In this definition $D\in (\R\cup\{-\infty\})^{n\times n}$ and $S$ is a binary matrix. However, we will always use $\sim_{ij}\hspace*{0.5ex}=\hspace*{0.5ex}\geq$ thus operating with zones that are closed in Euclidean topology. Hence, it is sufficient to represent a zone by matrix $D$. For the rest of the paper, the relation between a zone $Z$ and its DBM representation is expressed by: 
$D\in \dbm{Z}$ and $Z=\zone{D}$. 

Note that each $D\in (\R\cup\{-\infty\})^{n\times n}$ with all diagonal entries equal to $0$ corresponds to a unique (possibly empty) zone $\zone{D}$ for which $D$ is its DBM representation. 
In the expression $D\in \dbm{Z}$ we take into account that the same zone can have possibly infinitely many DBM representations, so that $\dbm{Z}$ denotes the set of such representations. However, there is a canonical DBM representation of a zone, and it can be found by means of max-plus or min-plus linear algebra. We will choose the max-plus representation (although we could have used the equivalent min-plus representation instead). 

\begin{defn}[Kleene star \cite{baccelli92,Butkovic2010,Heid2006} ]
For $A\in\Rmax^{n\times n}$, the (max-plus) Kleene star is defined as the following formal series: 
\begin{equation*}
A^\ast=\bigoplus_{k=0}^{\infty} A^{k}. \end{equation*}
\end{defn}

Note that this series converges and can be truncated as 
$$
A^\ast=\bigoplus_{k=0}^{n-1} A^{ k}
$$
if and only if $\lambda_n(A)\leq 0$. More precisely, the following facts for DBM representations and zones are known.
\begin{prop}[\cite{Mufid2018,MaxEig}]
If $D_1, D_2\in\dbm{Z}$ for a zone $Z$, then $D_1^\ast=D_2^\ast\in\dbm{Z}$. 
\end{prop}
\begin{prop}[\cite{Mine}]
Let $D\in(\R\cup\{-\infty\})^{n\times n}$ have $0$ diagonal and $Z=\zone{D}$. Then the following are equivalent:
\begin{itemize}
\item[(i)] $Z$ is non-empty; 
\item[(ii)] $\lambda_n(D)=0$.
\item[(iii)] $D^\ast=\bigoplus_{k=1}^{n-1} D^{ k}$; 
\end{itemize}
\end{prop}

Let us now formally introduce matrices which are closely related to zones, in max-plus algebra.

\begin{defn}[Strongly definite matrices~\cite{Butkovic2010}]
Let $D\in(\R\cup\{-\infty\})^{n\times n}$ have $0$ diagonal and $\lambda_n(D)=0$. Then, $D$ is called strongly definite. 
\end{defn}

The relation is the following~\cite{Butkovic2010, MaxEig}: If $D$ is a strongly definite matrix, then $\zone{D}$ is the same as the finite part of the max-plus eigenspace of $D$ associated with the eigenvalue $\lambda_n(D)$.  

The above propositions show that the tightest bounds for a zone are given by a uniquely defined Kleene star, which is the canonical DBM representation of this zone. These bounds are consistent if and only if this Kleene star or any DBM representation of the zone do not have any cycle with a positive cycle mean. When a zone $Z$ is represented by the inequalities derived from its canonical DBM, we will say that $Z$ is in canonical form.

We will also use that two zones can be conveniently intersected by means of taking the tropical sum of their DBM representations.
\begin{prop}[\cite{Mine,Mufid2018}]
\label{zones-meet}
    Let two zones $Z_1,Z_2$ in $\R^n$ with the DBM representations $D_1$ and $D_2$ be given. Then $D_1\oplus D_2\in \dbm{Z_1\cap Z_2}$.
\end{prop}

\if{
The canonical form of a zone $Z$, denoted as $Z^\ast$, is a zone with the tightest possible bounds \cite{here}.
\begin{prop}\cite{here}
    Given a zone $Z$ in $\R^n$ and its DBM representation $D$, the DBM representation for its canonical form is
    \begin{equation}
        \label{eq:DBM-canonical}
        \DBM{Z^\ast}=\bigoplus_{k=1}^{+\infty} D^{\otimes k}.
    \end{equation}
\end{prop}

\begin{prop}
    Given a zone $Z$ in $\R^n$ and its DBM representation $D$, $Z$ is not empty iff the RHS of \eqref{eq:DBM-canonical} is convergent, i.e,
    \[\bigoplus_{k=1}^{+\infty} D^{\otimes k} = \bigoplus_{k=1}^{n} D^{\otimes k}.\]
\end{prop}
}\fi

\begin{exmp}
    Suppose we have a zone in $\R^3$
    \[Z=\{[x_1~x_2~x_3]^\top\mid -2\leq x_1-x_2 \leq 4, 2\leq x_1-x_3\leq 3,0\leq x_2-x_3\leq 1\}.\]
    The corresponding DBM representation for $Z$ is
    \[
    D=\begin{bmatrix}
        0&-2&2\\
        -4&0&0\\
        -3&-1&0
    \end{bmatrix}.
    \]
   However $D$ is not the canonical DBM for $Z$ since
    \[
    D^\ast=D\oplus D^{\otimes 2} = \begin{bmatrix}
        0&1&2\\
        -3&0&0\\
        -3&-1&0
    \end{bmatrix}\neq D.
    \]
   $D^*$ yields the tightest inequalities for $Z$: 
    \[Z =\{[x_1~x_2~x_3]^\top\mid -1\leq x_1-x_2 \leq 3, 2\leq x_1-x_3\leq 3,0\leq x_2-x_3\leq 1\}.\]
\end{exmp}


\section{Over-approximation of maxmin-$\omega$ eigenspace}
\label{s:main}

\subsection{A sequence of over-approximations}
\label{ss:sequence}
In this section we will define a sequence of over-approximations of  maxmin-$\omega$ eigenspaces by zones. Let us formally denote by
\begin{equation}
\label{eq:image}
    \mathtt{Im}(A,D,\omega)=\{A\otimes_{\omega} x\mid x\in\zone{D}\}
\end{equation}
the image of $\zone{D}$ under $A\otimes_{\omega}$.
Instead of characterizing image~\eqref{eq:image}, our first aim is to find a zone $Z$ such that $\mathtt{Im}(A,D,\omega)\subseteq Z$: an over-approximation for image \eqref{eq:image}.

Given a strongly definite $D\in(\R\cup\{-\infty\})^{n\times n}$ and two row-vectors $a=[a_1~\cdots~a_n]$ and $b=[b_1~\cdots~b_n]$, we first show how to find the lower bound for
$a\otimes_{\omega} x - b\otimes_{\omega} x$
provided that $x\in\zone{D}$.

\begin{prop}
\label{prop:ax-bx-lower-min-max}
    For any strongly definite  $D\in(\R\cup\{-\infty\})^{n\times n}$, two row-vectors $a=[a_1~\cdots~a_n]$, $b=[b_1~\cdots~b_n]$, $\omega=p/n$ and $x\in\zone{D}$, we have
\begin{equation}
\label{eq:ax-bx-lower-min-max}
    a\otimes_{\omega} x - b\otimes_{\omega} x \geq \bigoplus_{(S_1,S_2)\in \mathcal{S}_1\times \mathcal{S}_2}\hspace*{-4ex}{}^\prime\hspace{2ex}\left\{\bigoplus_{(i,j)\in S_1\times S_2} \left\{a_i-b_j+D(i,j)\right\}\right\}
\end{equation}
where $\mathcal{S}_1$ and $\mathcal{S}_2$ are the set of all $p$-subsets and $(n+1-p)$-subsets of $[n]$, respectively.
\end{prop}
\begin{proof}
    Recall that due to \eqref{eq:bomega-min} and \eqref{eq:bomega-max}, we can express
    \[
    a\otimes_{\frac{p}{n}} x=\bigoplus_{S_1\in \mathcal{S}_1}\hspace*{-1ex}{}^\prime\left\{ \bigoplus_{i\in S_1} a_i+x_i\right\} ~\text{and}~
    b\otimes_{\frac{p}{n}} x=\bigoplus_{S_2\in \mathcal{S}_2}\left\{ \bigoplus_{j\in S_2}\hspace*{-0.5ex}{}^\prime ~b_j+x_j\right\}.
    \]
    Hence, we have 
    \begin{align*}
         a\otimes_{\frac{p}{n}} x - b\otimes_{\frac{p}{n}} x&=\bigoplus_{S_1\in \mathcal{S}_1}\hspace*{-1ex}{}^\prime\left\{ \bigoplus_{i\in S_1} a_i+x_i\right\}-\bigoplus_{S_2\in \mathcal{S}_2}\left\{ \bigoplus_{j\in S_2}\hspace*{-0.5ex}{}^\prime ~b_j+x_j\right\},\\
         &=\bigoplus_{S_1\in \mathcal{S}_1}\hspace*{-1ex}{}^\prime\left\{ \bigoplus_{i\in S_1} a_i+x_i\right\}+\bigoplus_{S_2\in \mathcal{S}_2}\hspace*{-1ex}{}^\prime\left\{ \bigoplus_{j\in S_2} -b_j-x_j\right\},\\
         &=\bigoplus_{(S_1,S_2)\in \mathcal{S}_1\times \mathcal{S}_2}\hspace*{-4ex}{}^\prime\hspace{2ex}\left\{\bigoplus_{(i,j)\in S_1\times S_2} \left\{a_i-b_j+x_i-x_j\right\}\right\},\\
         &\geq \bigoplus_{(S_1,S_2)\in \mathcal{S}_1\times \mathcal{S}_2}\hspace*{-4ex}{}^\prime\hspace{2ex}\left\{\bigoplus_{(i,j)\in S_1\times S_2} \left\{a_i-b_j+D(i,j)\right\}\right\}.
    \end{align*}
    The last assertion is due to the fact that $x_i-x_j\geq D(i,j)$ for each $i,j\in [n]$ and the monotonicity of all involved arithmetical operations.
\end{proof}


Consider a matrix 
\begin{equation}
\label{eq:L}
\left[
\begin{array}{c|c|c|c}
     a^\top&a^\top&\cdots&a^\top
\end{array}
\right]
-
\left[
\begin{array}{c}
     b  \\
     b\\
     \vdots\\
     b
\end{array}
\right]+D,
\end{equation}
where 
\[
\left[\begin{array}{c|c|c|c}
     a^\top&a^\top&\cdots&a^\top
\end{array}
\right],\quad 
\left[\begin{array}{c}
     b  \\
     b\\
     \vdots\\
     b
\end{array}
\right]
\]
denote, respectively, the matrix whose every column is equal to $a^{\top}$ and the matrix whose every row is equal to $b$.

Notice that the maximum operation in \eqref{eq:ax-bx-lower-min-max}, which is
\[
\bigoplus_{(i,j)\in S_1\times S_2} \left\{a_i-b_j+D(i,j)\right\},
\]
gives us the greatest element of the submatrix of~\eqref{eq:L} of size $p\times (n+1-p)$ extracted from the rows with indices in $S_1$ and the columns with indices in $S_2$. As a consequence, the lower bound for \eqref{eq:ax-bx-lower-min-max} can be expressed as
\begin{equation}
\label{eq:min-max-block}
    \mathtt{MinMaxValue}(D,a,b\mid p,n+1-p):=\min_{(S_1,S_2)\in \mathcal{S}_1\times\mathcal{S}_2}\left(\max_{i\in S_1,j\in S_2} D(i,j)+a_i-b_j\right), 
\end{equation}
where $\mathcal{S}_1$ denotes the set of all subsets of $[n]$ with $p$ elements and $\mathcal{S}_2$ denotes the sets of all subsets of $[n]$ with $n+1-p$ elements. We will also use the notation 
\begin{equation}
\label{eq:min-max-block-v2}
    \mathtt{MinMaxValue}(F\mid p,n+1-p):=\min_{(S_1,S_2)\in \mathcal{S}_1\times\mathcal{S}_2}\left(\max_{i\in S_1,j\in S_2} F(i,j)\right). 
\end{equation}

\begin{prop}
\label{p:minmaxvalue}
    For two matrices $A,B\in (\R\cup\{-\infty\})^{n\times n}$, the following conditions hold
    \begin{itemize}
        \item[(i)] If $A$ is a diagonal matrix (in a max-plus algebraic sense), then 
        \[\mathtt{MinMaxValue}(A\mid p,n+1-p)=\min\{A(1,1),\ldots,A(n,n)\}\]
        for each $p\in [n]$.
        \item[(ii)]  If $A\geq B$ then
        \[
        \mathtt{MinMaxValue}(A\mid p,n+1-p)\geq \mathtt{MinMaxValue}(B\mid p,n+1-p)
        \]
        for each $p\in [n]$.
    \end{itemize}
\end{prop}
\begin{proof}
(i): As $A$ is a max-plus diagonal matrix, all off-diagonal entries of $A$ are $-\infty$. To prove the claim it is sufficient to show that in any $p\times (n+1-p)$ submatrix of $A$ there is a diagonal entry, and that there exists a $p\times (n+1-p)$ submatrix of $A$, in which $A(i,i)$ is the only possibly finite entry (the rest being $-\infty$). For this, take any $p$ rows and consider the $p\times (n-p)$ submatrix where the set of column indices is the complement of the row indices (otherwise, it already contains a diagonal entry). Adding any column to it will result in a matrix containing a diagonal entry.  Furthermore, if we take $p$ rows that contain a given $i$, then to form the $p\times (n-p+1)$ submatrix that contains $A(i,i)$ we should add the $i$th column, and then in this column as well as in the resulting submatrix only $A(i,i)$ is possibly finite and the rest of the entries are equal to $-\infty$, as required.\\
(ii): obvious from the monotonicity of the arithmetical operations involved.    
\end{proof}

\begin{rem}
\label{r:diag}
Observe that if in \Cref{prop:ax-bx-lower-min-max} we take $a=b$, then we obtain that the right-hand side of~\eqref{eq:ax-bx-lower-min-max} is non-positive. Furthermore, it is actually equal to $0$ since any $p\times (n-p+1)$ submatrix of an $n\times n$ matrix contains a diagonal entry and $D(i,i)=0$.   
\end{rem}

Let us now for each $p$ define a sequence of matrices $\{D_p^k\}_{k\geq 0}$, where $D_p^0=I_n$ is an $n\times n$ max-plus identity matrix. The rest of the matrices are inductively defined by 
\begin{equation}
\label{e:dpk}
D_p^{k+1}(i,j):=\mathtt{MinMaxValue}(D_p^k,A_i,A_j\mid p,n+1-p)\quad \forall i,j\in [n] 
\end{equation}
where $A_i$ and $A_j$ denote the $i$th and $j$th rows of $A$.
We note that these are matrices with $0$ diagonal due to \Cref{r:diag}.

Denote $Z_p^k:=\zone{D_p^k}$ for all $p$ and $k$. Indeed, $Z_p^0=\zone{I_n}=\R[n]$. Due to \Cref{prop:ax-bx-lower-min-max} and \Cref{e:dpk} we obtain that
\begin{equation}
\label{e:image-incl}
\mathtt{Im}(A,D_{p}^k,p/n)\subseteq Z_p^{k+1},\quad \forall p\in[n], k\geq 0.
\end{equation}
In particular, zone $Z_p^{k+1}$ is non-empty showing, together with \Cref{r:diag}, that the following fact can be proved by a simple induction:
\begin{prop}
Matrices $D_p^k$ are strongly definite for any $p\in[n]$ and $k\geq 0$.
\end{prop}

\if{
Now given a DBM $\mathbb{E}_p^0=\R[n]$ and a matrix $A\in\R[n][n]$, we can compute the image approximation
\[
\mathbb{E}_{p}^{k+1}\supseteq \mathtt{Im}(A,\mathbb{E}_{p}^k,p/n),~~k=0,1,2,\ldots,
\]
Due to , the DBM $\mathbb{E}_{p}^{k+1}$ is given by
\[
\mathbb{E}_{p}^{k+1}(i,j)=\mathtt{MinMaxBlock}(\mathbb{E}_{p}^k+A_i^\top-A_j,p, n+1-p)
\]
where $A_i(l,\cdot)=A(i,\cdot)$ for each $i,l\in N$.
}\fi
We next establish the following properties of $D_p^k$.

\begin{prop}
\label{p:Dpk}
    For any $A\in\R^{n\times n}$, matrices $D_p^k$ have the following properties:
    \begin{itemize}
        \item[(i)] $D_1^1=D_2^1=\ldots =D_n^1$,
        \item[(ii)] $D_{p}^{k}\leq D_{p}^{k+1}$ for any $p\in[n]$ and $k\geq 0$.
    \end{itemize}
\end{prop}
\begin{proof}$\left.\right.$\\
    \begin{itemize}
        \item[(i)] We recall that $D_p^0=I$ and using Proposition~\ref{p:minmaxvalue}(i) we obtain that
        \begin{align*}
            D_{p}^1(i,j)&=\mathtt{MinMaxValue}(I,A_i,A_j\mid p, n+1-p)\\
            &=\min\{A(i,1)-A(j,1),\ldots,A(i,n)-A(j,n)\}.
        \end{align*}
        This implies that $D_1^1=D_2^1=\ldots =D_n^1$.
        \item[(ii)] For $k=0$, it is straightforward to see that $D_{p}^0\leq D_p^1$ for each $p\in [n]$. Suppose now that $D_p^k\leq D_{p}^{k+1}$.  Then due to Proposition \ref{prop:ax-bx-lower-min-max} (ii), we have
        \begin{align*}
            D_{p}^{k+2}(i,j)&= \mathtt{MinMaxValue}(D_{p}^{k+1},A_i,A_j\mid p,n+1-p)\\
            &\geq \mathtt{MinMaxValue}(D_{p}^k,A_i,A_j\mid p,n+1-p)\\
            &=D_{p}^{k+1}(i,j)
        \end{align*}
        which implies that
        $D_{p}^{k+2}\geq D_{p}^{k+1}$ thus proving the claim by induction.
    \end{itemize}
\end{proof}

\begin{cor}
\label{c:Zpk}
    For any $A\in\R^{n\times n}$, matrices $Z_p^k$ have the following properties:
    \begin{itemize}
        \item[(i)] $Z_1^1=Z_2^1=\ldots =Z_n^1$,
        \item[(ii)] $Z_{p}^{k+1}\subseteq Z_p^k$ for any $p\in[n]$ and $k\geq 0$.
    \end{itemize}
\end{cor}
\begin{proof}
 The first part follows from \Cref{p:Dpk} part (i) since $Z_p^1=\zone{D_p^1}$ for all $p\in [n]$ by definition, and the second part follows from \Cref{p:Dpk} part (ii) since $Z_p^k=\zone{D_p^k}$ for all $k\in [n]$ and as by \Cref{zones-meet} we have 
 $\zone{D_2}\subseteq \zone{D_1}$ whenever $D_2\geq D_1$ for any two strongly definite $D_1$ and $D_2$.
\end{proof}

\begin{prop}
\label{prop:stabilize}
     Given $A\in \R^{n\times n}$ with all entries integer, the 
      sequence of matrices $\{D_p^k\}_{k\geq 1}$ stabilizes after a finite number of steps.
\end{prop}
\begin{proof}
It can be observed that if $A$ has integer entries then the whole sequence of matrices $\{D_p^k\}_{k\geq 1}$ has integer entries. As $A$ has a finite eigenvector $x$ associated with $\omega=p/n$, this vector can be used to create a strongly definite $D_x$ with entries $D_x(i,j)=x_i-x_j$. This matrix canonically represents the zone $\{\lambda+x\mid \lambda\in\R\}$. Since this is a canonical DBM representation,
$D_x$ is an upper bound for all $D_p^k$, and since the matrices $D_p^k$ are increasing and have integer entries only, their sequence is stabilizing after a finite number of steps. 
\end{proof}

\begin{rem}
{\rm 
Proposition \ref{prop:stabilize} also applies when $A$ have rational entries.}
\end{rem}

The limit of the sequence $\{D_p^k\}$ will be denoted by $\overline{D}_p$ and the corresponding zone $\zone{\overline{D}_p}$ by $\overline{Z}_p$. We can use this zone to over-approximate the eigenspace $E_p(A)$

\begin{prop}
    Let $A\in \R^{n\times n}$, $p\in [n]$ and $\overline{Z}_p=\zone{\overline{D}_p}$. Then we have $E_p(A)\subseteq \overline{Z}_p$.
\end{prop}
\begin{proof}
We have $\mathtt{Im}(A,E_p(A),p/n)=E_p(A)$ by the definition of eigenvectors. Since $E_p(A)\subseteq Z_p^0=\R^n$, we obtain that 
\begin{equation*}
E_p(A)=\mathtt{Im}(A,E_p(A),p/n)\subseteq \mathtt{Im}(A,Z_p^0,p/n)\subseteq Z_p^1. 
\end{equation*}
By simple induction we obtain in a similar way that $E_p\subseteq Z_p^k$ for all $k\geq 0$ and any $p\in [n]$. Hence $E_p\subseteq \overline{Z}_p$.    
\end{proof}

In addition to the approximation for the eigenspace $E_p(A)$, it is also possible to obtain the bounds for the corresponding eigenvalue $\lambda_p$. These bounds result from the lower and upper bounds for
\begin{equation}
\label{eq:ax-xi}
    \bomega \{A(i,1)+x_1,\ldots,A(i,n)+x_n\} - x_i,~~i\in [n]
\end{equation}
where $x\in \overline{D}_p$.
\begin{prop}
\label{p:LUbounds}
    Let $A\in \R^{n\times n}$ and $p\in [n]$, and define the following two matrices:
    \begin{equation}
    \label{e:LU}
    L=A\otimes_\omega \overline{D}_p~\text{and}~ U=A\otimes_\omega (-\overline{D}_p)^\top, ~~\omega=p/n.
    \end{equation}
    Then, the eigenvalue $\lambda_p$ satisfies the following bounds:
    \[
    \max\{L(1,1),\ldots,L(n,n)\}\leq \lambda_p\leq \min\{U(1,1),\ldots,U(n,n)\}
    \]
\end{prop}
\begin{proof}
We first observe that, for all $i\in [n]$, we can obtain the following lower bound for \eqref{eq:ax-xi}, which becomes also a lower bound for $\lambda_p$ when we take $x\in E_p(A)$:
    \begin{align*}
        \lambda_p&=\bomega \{A(i,1)+x_1-x_i,\ldots,A(i,n)+x_n-x_i\}\\
        &\geq \bomega \{A(i,1)+\overline{D}_p(1,i),\ldots,A(i,n)+\overline{D}_p(n,i) \}\\
        &=A(i,\cdot)\otimes_{\omega} \overline{D}_p(\cdot,i)=L(i,i)
    \end{align*}
    As we have these inequalities for all $i$, we obtain $\lambda_p\geq \max_i L(i,i)$. We also obtain the following upper bound for all $i$: 
    \begin{align*}
    \lambda_p&=\bomega \{A(i,1)+x_1-x_i,\ldots,A(i,n)+x_n-x_i\}\\
    &\leq\bomega \{A(i,1)-\overline{D}_p(i,1),\ldots,A(i,n)-\overline{D}_p(i,n) \}\\
        &=A(i,\cdot)\otimes_{\omega} (-\overline{D}_p(i,\cdot))^\top=U(i,i),
    \end{align*}
    and hence $\lambda_p\leq \min_i U(i,i)$. Thus we obtain the desired lower and upper bounds for $\lambda_p$.
\end{proof}

\if{
\begin{cor}
    Let $A\in \R[n][n]$ and $p\in [n]$. Let the matrix $L$ as in~\Cref{e:LU}. If there exists $\alpha$ such that $L=\alpha\otimes \overline{D}_p$ then $\lambda_p(A)=\alpha$ and $E_p(A)=\zone{\overline{D}_p}=\overline{Z}_p$. 
\end{cor}
\begin{pf}
In this case, every column of $\overline{D}_p$ is an eigenvector of $A$ associated with $\omega$, and $\alpha$ is the corresponding eigenvalue. Furthermore, since $\overline{Z}_p$ is an over-approximation of $E_p(A)$ but any vector in $\overline{Z}_p$ belongs to $E_p(A)$, we have $E_p(A)=\overline{Z}_p(A)$.
\end{pf}
}\fi

\if{
\begin{prop}
    Given a matrix $A\in \R[n][n]$ and $p\in [n]$, there exists a matrix $B\in \Rmax[n][n]$ such that $\lambda_p(A)=\lambda_n(B)$ and $E_p(A)=E_n(B)$
\end{prop}
}\fi
\begin{exmp}
\label{ex:Ds}
    Suppose we have a matrix
    \[A=
    \begin{bmatrix}
    4&7&10&2\\
    9&10&2&0\\
    10&9&7&2\\
    9&10&7&1
    \end{bmatrix}.
    \]
    For $\omega=1/2$, the corresponding image approximation sets are
    \[
    D_2^0=\begin{bmatrix}
        0&-\infty&-\infty&-\infty\\
        -\infty&0&-\infty&-\infty\\
        -\infty&-\infty&0&-\infty\\
        -\infty&-\infty&-\infty&0
    \end{bmatrix},
    D_2^1=\begin{bmatrix}
        0&-5&-6&-5\\
        -8&0&-5&-5\\
        -3&-1&0&-1\\
        -3&0&-1&0
    \end{bmatrix}
    \]
    \[
    D_2^2=\begin{bmatrix}
        0&-1&-6&-5\\
        -5&0&-5&-5\\
        0&2&0&-1\\
        0&1&-1&0
    \end{bmatrix},
     D_2^3=\begin{bmatrix}
        0&-1&-6&-5\\
        -2&0&-5&-5\\
        2&2&0&-1\\
        2&1&-1&0
    \end{bmatrix}
    \]
    \[
     D^4_2=\begin{bmatrix}
        0&-1&-6&-5\\
        -1&0&-5&-5\\
        2&2&0&-1\\
        2&1&-1&0
    \end{bmatrix},
        D^5_2=\begin{bmatrix}
        0&-1&-6&-5\\
        -1&0&-5&-5\\
        2&2&0&-1\\
        2&1&-1&0
    \end{bmatrix}
    \]
    Notice that $D^4_2=D^5_2=\overline{D}_{2}$. 
    Furthermore, we have
    \[
    L= A\otimes_\omega \overline{D}_2=\begin{bmatrix}
        4&3&1&2\\
        4&4&2&1\\
        8&9&4&4\\
        9&8&3&4
    \end{bmatrix}~\text{and}~
    U=A\otimes_\omega (-\overline{D}_2)^\top=
    \begin{bmatrix}
        7&7&3&2\\
        8&7&2&3\\
        10&9&7&8\\
        9&10&7&7
    \end{bmatrix}.
    \]
    Hence, $4\leq \lambda_2\leq 7$.
\if{    
    Note that the actual eigenvalue is $\lambda_2=5.5$ and $x=[0,\; 1,\; 4.5,\; 3.5]^{\top}$ is an eigenvector corresponding to this eigenvalue. This eigenvector corresponds to a DBM 
    \begin{equation*}
        D_x=
        \begin{bmatrix}
            0 & -1 & -4.5 & -3.5\\
            1 & 0 & -3.5 & -2.5\\
            4.5 & 3.5 & 0 & 1\\
            3.5 & 2.5 & -1 & 0
        \end{bmatrix},
    \end{equation*}
This DBM is an upper bound for the sequence $\{D_2^k\}_{k\geq 1}$, so this sequence has to stabilize in a finite number of steps.
}\fi
\end{exmp}

\newpage
The procedure for computing $\overline{D}_p$ is summarized in~\Cref{a:sequence}.

\begin{algorithm}
\caption{Computing $\overline{D}_p$}
\label{a:sequence}
{\rm
Input: $A\in \R^{n\times n}, p\in [n]$\\
Output: $\overline{D}_p$
\begin{itemize}
\item[1.] {\em Initialisation.} Set $D_p^0=I_n$ (the $n\times n$ max-plus identity matrix). \\ Set $k:=1$. 
\item[2.] {\em Image Approximation.} For each $i,j\in [n]$ compute \[ D_p^{k}(i,j):=\mathtt{MinMaxValue}(D_p^{k-1},A_i,A_j\mid p,n+1-p)\]
\item[3.] {\em Stability Checking.} If $D_p^k=D_p^{k-1}$ return $\overline{D}_p=D_p^{k}$. 
\item[4.] Else set $k:=k+1$ and go to 2.
\end{itemize}
}
\end{algorithm}
\begin{rem}
\label{rem:Dp}
{\rm In general, given any zone $Z\subseteq \R^n$, we can try to check whether $Z\cap E_p(A)\neq \emptyset$ by using the following modification of Algorithm 1. We start by setting $D_p^0$ to the canonical DBM of $Z$ on Step 1 (instead of $I_n$). Then, on Step 2, we not only compute $D_p^k$ but also check whether it is strongly definite. If it is not then we can terminate the algorithm. In case if the algorithm terminates due to Step 3, the resulting zone is an over-approximation of $Z\cap E_p(A)$.}
\end{rem}

\subsection{Active and inactive entries}
\label{ss:active}

Let $x\in\R^n$. If an entry $(i,j)$ satisfies $(A\otimes_{\omega}x)_i=A(i,j)+x_j$, then $(i,j)$ is called an active entry of $A$ with respect to $x$ and $\omega$, and an inactive entry otherwise.

\begin{prop}
\label{LAUactive}
    Let $A\in \R^{n\times n}$, $p\in N$, $x\in \overline{Z}_p$, and let $L$ and $U$
    be two matrices defined by~\eqref{e:LU}.
    If  $i,j\in [n]$ satisfy $A(i,j)=L(i,j)=U(i,j)$, then $(i,j)$ is an active entry of $A$ with respect to $\omega$ and $x$.
\end{prop}
\begin{proof}
Since $x\in\overline{Z}_p$, we have
\begin{align}
\nonumber
        (A\otimes_{\omega} x)_i-x_j&=\bomega \{A(i,1)+x_1-x_j,\ldots,A(i,n)+x_n-x_j\}\\
        \nonumber
        &\geq \bomega \{A(i,1)+\overline{D}_p(1,j),\ldots,A(i,n)+\overline{D}_p(n,j) \}\\
        \label{eq:AL}
        &=L(i,j)
    \end{align}
    and, similarly,
    \begin{align}
    \nonumber
     (A\otimes_{\omega} x)_i-x_j&=\bomega \{A(i,1)+x_1-x_j,\ldots,A(i,n)+x_n-x_j\}\\
     \nonumber
    &\leq\bomega \{A(i,1)-\overline{D}_p(j,1),\ldots,A(i,n)-\overline{D}_p(j,n) \}\\
    \label{eq:AU}
        &=U(i,j).
    \end{align}
Therefore, if $L(i,j)=U(i,j)$ then 
\begin{equation}
  \label{e:LUactive}
   L(i,j)= (A\otimes_{\omega} x)_i-x_j= U(i,j). 
\end{equation}
 By contradiction, assume that $(i,j)$ is an inactive entry. However, then $A(i,j)\neq (A\otimes_{\omega} x)_i-x_j$, which contradicts \Cref{e:LUactive}. Thus $(i,j)$ is an active entry.
\end{proof}

\begin{defn}
\label{def:active-element}
    {\rm Let a matrix $A\in \R^{n\times n}$ and $\omega=p/n$ for $p\in [n]$ be given. The {\em strongly active matrix} with respect to $A$ and $\omega$ is an $n\times n$ matrix whose elements are
    \[
    T(i,j)=\left\{\begin{array}{ll}
         A(i,j), &  ~\text{if}~A(i,j)=L(i,j)=U(i,j),\\
         \infty, & \text{otherwise},
    \end{array}\right.
    \]
    where $L$ and $U$ are defined by \eqref{e:LU}. The {\em strongly active graph} with respect to $A$ and $\omega$, denoted by $\Act(A,\omega)$, is a directed graph whose node set is $[n]$ and whose arcs are $(i,j)$ whenever $A(i,j)=L(i,j)=U(i,j)$.}
\end{defn}
In what follows we assume that if $T$ is multiplied by a vector in the max-plus sense (respectively, in the min-plus sense) then the ``$\infty$'' entries are treated as $-\infty$'s (respectively, as $+\infty$'s). 

\if{
In case there are at least one finite element in each row of $T$ then we can easily find the eigenvalue $\lambda_p(A)$ from the resulting $T$ and $\act[A][p/n]$. The eigenvalue $\lambda_p$ is indeed the cycle mean in $\act[A][p/n]$ Furthermore, $\lambda_p(A)=\lambda_1(T)=\lambda_n(T)$ where $\lambda_1(T)$ (resp. $\lambda_n(T)$) is computed by treating the infinite element as $+\infty$ (resp. $-\infty$).
}\fi

\begin{prop}
\label{prop:active-element}
    Let $A\in \R^{n\times n}$ and $p\in [n]$. Suppose that $T$ and
    $\Act(A,p/n)$ be, respectively, the strongly active  matrix and the strongly active graph with respect to $A$ and $p/n$. If there is at least one finite element in each row of $T$, then:
\begin{itemize}
    \item[(i)] $\lambda_p(A)$ is the only cycle mean in $\Act(A,p/n)$,
    \item[(ii)] for each $x\in \overline{Z}_p$, $A\otimes_{p/n}x = T\otimes x=T\otimes^\prime x$, 
    \item[(iii)] for each $x\in E_p(A)$, $\Act(A,p/n)$ is a subgraph of $\Sat(A,p/n,x)$.
\end{itemize}
\end{prop}
\begin{proof}
    (i) The condition that all rows of $T$ contain a finite element ensures the existence of the cycle mean. Each cycle $(i_1,i_2,\ldots,i_k)$ in $\Act(A,p/n)$ has the cycle mean equal to $\lambda_p(A)=(A(i_1,i_2)+\ldots+A(i_k,i_1)/k$, since all the arcs in this cycle are active for any $x\in E_p(A)$.\\  
    (ii) Suppose there are $k>0$ finite elements in each row of $T$: $(i,j_1),\ldots, (i,j_k)$ for $i\in [n]$. By \Cref{LAUactive}, 
    \begin{equation*}
    \begin{split}
    (A\otimes_{p/n}x)_i &=A(i,j_1)+x_{j_1}=\cdots= A(i,j_k)+x_{j_k}\\
    &=\max\{A(i,j_1)+x_{j_1},\ldots,A(i,j_k)+x_{j_k}\}\\
    &=\min\{A(i,j_1)+x_{j_1},\ldots,A(i,j_k)+x_{j_k}\}.
    \end{split}
    \end{equation*}
    (iii) Suppose that $(i,j)\in \Act(A,p/n)$. Then, by \Cref{LAUactive}, $(A\otimes_{p/n} x)_i= A(i,j)+x_j$ for all $x\in \overline{D}_p$. Furthermore, since the eigenvalue is already known, we have $(A\otimes_{p/n} x)_i=x_i+\lambda_p$ for $x\in E_p(A)$. The inclusion relation $E_p(A)\subseteq \zone{\overline{D}_p}$ implies $(i,j)\in \Sat(A,p/n,x)$.
\end{proof}

The following result indicates the direct consequence of finding strongly active element matrix $T$ whose all rows contain finite elements. In this case it follows that the max-plus and min-plus linear combinations of eigenvectors are also eigenvectors. This feature in \Cref{cor:eigen-combin} may not hold if there is a row where $T$ does not contain a finite entry, as shown in \Cref{exmp:interesting}.
\begin{cor}
\label{cor:eigen-combin}
    Let $A\in \R^{n\times n}$ and $p\in [n]$. Suppose that $T$ is the strongly active matrix. If there are at least one finite element in each row of $T$, then for each $x,y\in E_p(A)$ we have $(\alpha \otimes x)\oplus (\beta \otimes y)\in E_p(A) $ and $(\alpha \otimes x)\oplus^\prime (\beta \otimes y)\in E_p(A)$ for each $\alpha,\beta\in \R$.
\end{cor}
\begin{proof}
    Since $\overline{Z}_p$ is a zone, it is stable both with respect to max-plus linear combinations and with respect to min-plus linear combinations. Thus if $x,y\in \overline{Z}_p$ then $(\alpha \otimes x)\oplus (\beta \otimes y)\in \zone{\overline{D}_p}$ and $(\alpha \otimes x)\oplus^\prime (\beta \otimes y)\in \zone{\overline{D}_p}$ for each $\alpha,\beta\in\R$. Due to \Cref{LAUactive} and \Cref{prop:active-element} part (ii)
    \begin{align*}
    &A\otimes_{\omega} ((\alpha \otimes x)\oplus (\beta \otimes y))=T\otimes ((\alpha \otimes x)\oplus (\beta \otimes y))   \\
    &=(\alpha \otimes T\otimes x) \oplus (\beta\otimes T\otimes y)=\lambda_p+(\alpha\otimes x\oplus\beta\otimes y),
    \end{align*}
    and the min-plus linearity is shown similarly.
\end{proof}

The following example shows the use of \Cref{LAUactive} and \Cref{prop:active-element}.

\begin{exmp}
\label{ex:activeelts}
    Based on \Cref{ex:Ds}, we have $A(i,j)=L(i,j)=U(i,j)$ for $(i,j)\in \{(1,4),(2,3),(3,2),(4,1)\}$. Hence for each $x\in \overline{Z_2}$
    we obtain
    \[
    A\otimes_\omega x=\begin{bmatrix}
        x_4+2\\x_3+2\\x_2+9\\x_1+9
    \end{bmatrix}=
    \begin{bmatrix}
        \cdot&\cdot&\cdot&2\\
        \cdot&\cdot&2&\cdot\\
        \cdot&9&\cdot&\cdot\\
        9&\cdot&\cdot&\cdot\\
    \end{bmatrix}
    \otimes 
    \begin{bmatrix}
        x_1\\x_2\\x_3\\x_4
    \end{bmatrix}
    \]
    From this, the only possible eigenvalue is $\lambda=11/2$. Furthermore, the eigenspace can be over-approximated as
    \[
    E_2(A)\subseteq \zone{\overline{D}_2}\cap \{x\in \R^4 \mid x_4-x_1=7/2,x_3-x_2=7/2\}.
    \]
\end{exmp}
\if{Notice that, from Example 2.24, we already found the active element in each row $A$ and then we can construct the so-called active element matrix $I$. From this matrix, we can find the eigenvalue $\lambda_p(A)=\lambda_1(I)=\lambda_n(I)$ and the ``candidate'' of eigenspace
Note that if we, like in \Cref{ex:activeelts}, manage to find an active entry in every row, then we can characterize the eigenspace as the zone defined by
\begin{equation}
\label{e:eigenzone}
E_p(A)\subseteq\overline{Z}_p\cap \{x\mid A(i,j)+x_j=\lambda_p+x_i,\; (i,j)\; \text{is active}\}.
\end{equation}
}\fi
The previous example suggests the following over-approximation of $E_p(A)$ also in general case:
\begin{equation}
\label{e:eigenzone}
E_p(A)\subseteq\overline{Z}_p\cap \{x\mid A(i,j)+x_j=\lambda_p+x_i,\; (i,j)\in\Act(A,p/n)\}.
\end{equation}

\if{
However, it is possible that from matrices $A,L,U$, we are not able to find the active element matrix $I$ yet. This is because, there may be not active element in a row of $A$.
}\fi

The entries $(i,j)$ for which the corresponding arc belongs to $\Act(A,p/n)$ are guaranteed to be active. We next search for conditions that can eliminate some inactive entries. We start by presenting a bound on $\lambda_p$ resulting from $(i,j)$ being an active entry of $A$ with respect to 
$x\in E_p(A)$.
\begin{prop}
\label{p:PQbound}
     Let $A\in \R^{n\times n}$ and $p\in [n]$. If $(i,j)$ is an active entry of $A$ with respect to some $x\in E_p(A)$ and $\omega=p/n$, then 
     $$
A(i,j)+\overline{D}_p(j,i)\leq \lambda_p\leq A(i,j)-\overline{D}_p(i,j).
$$
\end{prop}
\begin{proof}
    Suppose that $(i,j)$ is an active element with respect to some $x\in E_p(A)$. Then we have $x_j+A(i,j)=\lambda_p+x_i$ which implies $\lambda_p=A(i,j)+x_j-x_i$. Since $x\in\zone{\overline{D}_p}$, we have $\overline{D}_p(j,i)\leq x_j-x_i\leq -\overline{D}_p(i,j)$. Adding $A(i,j)$ to all sides of these inequalities and substituting $\lambda_p=A(i,j)+x_j-x_i$ completes the proof.
\end{proof}

Combining this bound with the bound using $L$ and $U$ matrices defined in \Cref{e:LU}, the next corollary describes a case when we can conclude that a given entry is inactive.
\begin{prop}
\label{ALUinactive}
    Let $A\in \R^{n\times n}$ and $p\in [n]$. Let the matrices $L$ and $U$ be as in~\Cref{e:LU}.
     Then, any $(i,j)$ that satisfies $A(i,j)<L(i,j)$ or $A(i,j)>U(i,j)$ is inactive with respect to $\omega=p/n$ and any $x\in E_p(A)$.
\end{prop}
\begin{proof}
    By \Cref{eq:AL} we have $(A\otimes_{\omega}x)_i-x_j\geq L(i,j)$ for each $x\in \overline{Z}_p$. The condition $A(i,j)<L(i,j)$ implies that $(A\otimes_{\omega}x)_i>A(i,j)+x_j$. Similarly, by \Cref{eq:AU}, if $A(i,j)>U(i,j)$ then $(A\otimes_{\omega}x)_i<A(i,j)+x_j$.
    Hence, for either case, $(i,j)$ is inactive for each $x\in \overline{Z}_p$. The inclusion relation $E_p(A)\subseteq \overline{Z}_p$ completes the proof.
\end{proof}
\begin{prop}
\label{LUPQinactive}
    Let $A\in R^{n\times n}$ and $p\in [n]$. Let the matrices $L$ and $U$ be as in~\Cref{e:LU} and let 
    $l=\max\{L(1,1),\ldots, L(n,n)\}$ and $u=\min\{U(1,1),\ldots,U(n,n)\}$.
    Furthermore, define $P$ and $Q$ by
    \begin{equation}
    \label{e:PQ}
    P=A+\overline{D}_p^\top,~\text{and}~Q=A-\overline{D}_p,
    \end{equation}
     Then, any $(i,j)$ that satisfies $P(i,j)>u$ or $Q(i,j)<l$ is inactive with respect to $\omega=p/n$ and any $x\in E_p(A)$.
\end{prop}
\begin{proof}
By \Cref{p:LUbounds} we have $l\leq \lambda_p\leq u$. 
By contradiction, assume that $(i,j)$ is active for some $x\in E_p(A)$. Then by \Cref{p:PQbound} we also have $P(i,j)\leq\lambda_p\leq Q(i,j)$. But this is not possible when we have $P(i,j)>u$ or $Q(i,j)<l$.      
\end{proof}

Following \Cref{def:active-element}, we also define a possibly active element matrix.

\begin{defn}
\label{def:possibly-active}
    {\rm
    Let a matrix $A\in \R^{n\times n}$ and $\omega=p/n$ for $p\in [n]$ be given. The {\em possibly active matrix} with respect to ${A}$ and $\omega$ is an $n\times n$ matrix whose entries are
    \[
    \hat{A}(i,j)=\left\{\begin{array}{ll}
            &  ~\text{if}~A(i,j)<L(i,j)~\text{or}~ A(i,j)>U(i,j) ~\text{or}\\
         \infty, &  ~~~~\text{or}~ P(i,j)>\min_{i\in [n]}\{U(i,i)\}\\
         &   ~~~~\text{or}~ Q(i,j)<\min_{i\in [n]}\{L(i,i)\},\\
         A(i,j), & \text{otherwise},
    \end{array}\right.
    \]
    where $L$ and $U$ are defined by \eqref{e:LU} and $P$ and $Q$ are defined by \eqref{e:PQ}. The {\em possibly active graph} with respect to $A$ and $\omega$ denoted by $\Pos(A,\omega)$, is a directed graph whose node set is $[n]$ and  the arcs of which are $(i,j)$ whenever $\hat{A}(i,j)$ is finite.}
\end{defn}

\begin{exmp}
\label{ex:refinement}
\normalfont
    Suppose we have a matrix
    \[
    A=\begin{bmatrix}
        14 & 2 &18&  5\\
        5 & 2 & 3 &14\\
        1& 13& 12&  0\\
        19 & 6 &14& 12\\
    \end{bmatrix}.
    \]
    For $p=2$ we compute
    \begin{center}
    $\overline{D}_2\!=\!\begin{bmatrix}
        0&6&3&-5\\
        -12&0&-7&-11\\
        -5&1&0&-10\\
        -1&11&4&0
    \end{bmatrix}\!,
    L\!=\!\begin{bmatrix}
        4&16&9&5\\
        -2&4&3&-7\\
        1&11&4&0\\
        9&15&14&4
    \end{bmatrix}\!, U\!=\!\begin{bmatrix}
        10&16&15&5\\
        0&10&3&-1\\
        5&13&10&2\\
        11&21&14&10
    \end{bmatrix}$\\
    $P=\begin{bmatrix}
        14&-10&13&4\\
        11&2&4&25\\
        4&6&12&4\\
        14&-5&4&12
    \end{bmatrix},
    Q=\begin{bmatrix}
        14&-4&15&10\\
        17&2&10&25\\
        6&12&12&10\\
        20&-5&10&12
    \end{bmatrix}.$
    \end{center}
    
    Using $A,L$ and $U$ and \Cref{LAUactive}, entries $(1,4),(2,3)$ and $(4,3)$ are active with respect to any $x\in\overline{Z}_2=\zone{\overline{D}_2}$. We can also compute $l=4$ and $u=10$ and find some inactive entries of $A$ using \Cref{LUPQinactive}.
    This gives us the following matrix of possibly active entries of $A$
    \[
    \hat{A}=\begin{bmatrix}
        \cdot & \cdot &\cdot&  \underline{5}\\
        \cdot & \cdot & \underline{3} &\cdot\\
        1& 13& \cdot&  0\\
        \cdot & \cdot &\underline{14}& \cdot\\
    \end{bmatrix}.
    \]
    The entries which are known to be active are underlined.
\end{exmp}

The following proposition describes the consequence of finding such inactive elements w.r.t. the lower and upper bounds for the eigenvalue $\lambda_p$.
\begin{prop}
 \label{prop:possibly-active-element}
    Let $A\in \R^{n\times n}$ and $p\in [n]$. Let $\hat{A}$ and
    $\Pos(A,p/n)$ be, respectively, the possibly active matrix and the possibly active graph with respect to $A$ and $p/n$. Then 
\begin{itemize}
    \item[(i)] for each $x\in E_p(A)$, $\Sat(A,p/n,x)$ is a subgraph of $\Pos(A,p/n)$,
    \item[(ii)] $\lambda_p(A)$ is a cycle mean in $\Pos(A,p/n)$,
   \item[(iii)] $\lambda_1(A)\leq \lambda_1(\hat{A})\leq \lambda_p(A)\leq \lambda_n(\hat{A})\leq \lambda_n(A)$,
\end{itemize}
\end{prop}
\begin{proof} (i): Follows from \Cref{ALUinactive} and \Cref{LUPQinactive}.\\
Note that, due to this inclusion, there is at least one finite entry in each row of $\hat{A}$. Therefore $\Pos(A,p/n)$ has at least one cycle. In particular, the maximum cycle mean $\lambda_n(\hat{A})$ and the minimum cycle mean $\lambda_1(\hat{A})$ are both finite (assuming that only finite cycle means are considered both in the minimization and in the maximization).\\
(ii) Follows from part (i) and since $\lambda_p(A)$ is a cycle mean in $\Sat(A,p/n,x)$ by \Cref{p:satcycle}.
(iii): The inequalities $\lambda_1(A)\leq \lambda_1(\hat A)$ and $\lambda_n(\hat{A})\leq \lambda_n(A)$ follow by the monotonicity of $\lambda_1$ and $\lambda_n$, and the inequalities $\lambda_1(\hat{A})\leq \lambda_p(A)\leq \lambda_n(\hat{A})$ are due to part (ii).
\end{proof}

\subsection{Refinement Procedure}
\label{ss:refinement}

We will now present a heuristic refinement procedure for $\overline{Z}_p$ based on the observations made in the previous subsection.

As shown in \Cref{ex:refinement}, the stabilized zone $\overline{Z}_p$ and its DBM representation $\overline{D}_p$ may give us an ``incomplete'' active element matrix. However, we already know that the eigenvalue $\lambda_p(A)$ satisfies $l\leq \lambda_p(A)\leq u$ where $l=\max\{L(1,1),\ldots,L(n,n)\}$ and $u=\min\{U(1,1),\ldots,U(n,n)\}$ for $L$ and $U$ defined by \eqref{e:LU}. The main idea for our refinement procedure is to check whether $\lambda_p(A)\geq m$ or $\lambda_p(A)\leq m$ where $m=(l+u)/2$. For this we define the following zones: 
\begin{equation}
\label{e:eigen-refine1}
X=\overline{Z}_p\cap \{x\mid A(i,j)+x_j\geq m+x_i,\; T(i,j)\; \text{is finite}\}.
\end{equation}
and
\begin{equation}
\label{e:eigen-refine2}
Y=\overline{Z}_p\cap \{x\mid A(i,j)+x_j\leq m+x_i,\; T(i,j)\; \text{is finite}\}.
\end{equation}
If the over-approximation of $X\cap E_p(A)$ computed by \Cref{rem:Dp} is empty, then we can conclude that $\lambda_p(A)<m$ and set $\overline{Z}_p$ equal to the over-approximation of $Y\cap E_p(A)$ (computed by the same method). If the over-approximation of $Y\cap E_p(A)$ is empty, then we can conclude that $\lambda_p(A)>m$ and set $\overline{Z}_p$ equal to the over-approximation of $X\cap E_p(A)$. If both approximations are non-empty then we can take their intersection as an over-approximation of $E_p(A)$ and have to exit with a conjecture that $\lambda_p(A)=m$. When this is not the case, we update the matrices $U,L,P,Q,T,\hat{A}$ using the new over-approximation, and then update the bounds $l$ and $u$ on $\lambda_p(A)$ using the minimum and maximum cycle means of $\hat{A}$ (see \Cref{prop:possibly-active-element}(iii)). Then we compute $m=(l+u)/2$, redefine $X$ and $Y$ using \Cref{e:eigen-refine1} and \Cref{e:eigen-refine2} and repeat the above.

\if{
To show that $\lambda_p(A)\geq m $, with the help of active element matrix $T$, one needs to check the non-emptiness of over-approximation $X\cap E_p(A)$ where
\begin{equation}
\label{e:eigen-refine1}
X=\overline{Z}_p\cap \{x\mid A(i,j)+x_j\geq m+x_i,\; T(i,j)\; \text{is finite}\}.
\end{equation}
Similarly, the condition $\lambda_p(A)\leq m $ is equivalent to the non-emptiness of over-approximation $Y\cap E_p(A)$ where
\begin{equation}
\label{e:eigen-refine2}
Y=\overline{Z}_p\cap \{x\mid A(i,j)+x_j\leq m+x_i,\; T(i,j)\; \text{is finite}\}.
\end{equation}
}\fi


The refinement procedure is formally summarized in \Cref{alg:refinement}. 

\begin{algorithm}[!ht]
\caption{Refinement procedure}
\label{alg:refinement}
{\rm
Input: $A\in \R^{n\times n},~p\in [n]$, stabilized DBM $\overline{D}_p$, strongly active element\\
\hspace*{7ex}matrix $T$, possibly active element matrix $\hat{A}$\\
Output: refined DBM $\overline{D}_p$, eigenvalue $\lambda_p(A)$, strongly active element\\
\hspace*{8.5ex} matrix $T$, possibly active element matrix $\hat{A}$
\begin{enumerate}
\item Compute matrices $L$ and $U$ by \eqref{e:LU}; and $P$ and $Q$ by \eqref{e:PQ}
\item Compute $l:=\max_{i\in [n]}\{L(i,i)\}$ and $u=\min_{i\in [n]}\{U(i,i)\}$
\item Set $m: = (l+u)/2$
\item Set $B:=\overline{D}_p$ and $C:=\overline{D}_p$
\item For each $(i,j)\in [n]\times [n]$ do\\
If $T(i,j)$ is finite; then $B(j,i):=\max\{B(j,i),m-A(i,j)\}$\\ $C(i,j):=\max\{C(i,j),A(i,j)-m \}$
\item Compute $\overline{B}$ and $\overline{C}$, the DBM representation for the over-approximation of $\zone{B}\cap E_p(A)$ and $\zone{C}\cap E_p(A)$ respectively (see \Cref{rem:Dp})
\item  If $\overline{B}$ and $\overline{C}$ contain no positive diagonal element, then\\
$\overline{D}_p:=(\overline{B}\oplus \overline{C})^\ast, ~l:=m, ~u:=m$\\
 Else if $\overline{B}$ contains positive diagonal elements, then $\overline{D}_p:=\overline{C}, ~u:=m$\\
 Else if $\overline{C}$ contains positive diagonal elements, then $\overline{D}_p:=\overline{B}, ~l:=m$
\item Update matrices $U,L,P$ and $Q$ according to the new $\overline{D}_p$
\item Update matrices $T$ and $ \hat{A}$ according to the new $U,L,P$ and $Q$
\item For each $i\in [n]$ do\\
If $T(i,\cdot)$ contains finite element; then $\hat{A}(i,\cdot):=T(i,\cdot)$ \\    
Else if $\hat{A}(i,\cdot)$ has only one finite element; then $T(i,\cdot):=\hat{A}(i,\cdot)$

\item $l:=\max\{l,\lambda_1(\hat{A})\}$, $u:=\min\{u,\lambda_n(\hat A)\}$
\item If $l=u$; then return $\overline{D}_p,l,T,\hat{A}$\\
Else go to Step 3
\end{enumerate}
}
\end{algorithm}

\begin{exmp}
{\rm From the preceding \Cref{ex:refinement}, based on the values of $l$ and $u$, one can compute $m=7$. The assumption of $\lambda_2(A)\geq 7$ and $\lambda_p(A)\leq 7$ respectively yield the following zones (represented as DBMs). The underlined entries are due to \Cref{e:eigen-refine1,e:eigen-refine2}.

\[
\begin{bmatrix}
        0&6&3&-5\\
        -12&0&-7&-11\\
        -5&\underline{4}&0&\underline{-7}\\
        \underline{2}&11&4&0
    \end{bmatrix}
    \in\dbm{X},
    \begin{bmatrix}
        0&6&3&\underline{-2}\\
        -12&0&\underline{-4}&-11\\
        -5&1&0&-10\\
        -1&11&\underline{7}&0
    \end{bmatrix}
    \in\dbm{Y}.
\]
    One can check that both over-approximation $\overline{X}$ of $X\cap E_2(A)$ and over-approximation  $\overline{Y}$ of $Y\cap E_2(A)$ are non-empty; the refined over-approximation for the eigenspace is
\[
\overline{D}_2:=
 \begin{bmatrix}
0 & 9 & 5 & -2\\
 -9 & 0 &-4 &-11\\
 -5 & 4 & 0  &-7\\
  2 &11 & 7 &  0
    \end{bmatrix}\in\dbm{\overline{X}\cap \overline{Y}}
\]
Using $\overline{D}_2$ the updated strongly active matrix is
\[
T=
\begin{bmatrix}
      \cdot & \cdot &\cdot&  {5}\\
        \cdot & \cdot &{3} &\cdot\\
       \cdot & \cdot& \cdot&  0\\
        \cdot & \cdot &{14}& \cdot\\
\end{bmatrix}.
\]
Furthermore, we also obtain $\lambda_2=m=7$. Notice that, all columns of $\overline{D}_2$ are the multiple of the first one. One could also check that the first column of $\overline{D}_2$ is indeed a maxmin-$1/2$ eigenvector associated with $\lambda_2=7$.

}
\end{exmp}
\begin{rem}
    {\rm \Cref{alg:refinement} depends on the existence of finite entries in $T$ as per \Cref{e:eigen-refine1,e:eigen-refine2}. If all elements of $T$ are $\infty$, then in Steps 4-7 we will have $B=C=\overline{D}_p$ and $l=u=m$ (which is not always $\lambda_p(A)$). For this reason, \Cref{alg:refinement} is still a heuristic one and might not give us a refinement of $E_p(A)$. Hence, the algorithm is terminated once the eigenvalue $\lambda_p(A)$ is known or conjectured.}
\end{rem}

It should be noted that \Cref{alg:refinement} may fail to yield a strongly active matrix $T$ whose all rows contain finite entries. Interestingly, this may happen when there are multiple cycles in $\Act(A.\omega)$ whose average weight equals $\lambda_p$. As a result, one could obtain multiple essentially different eigenvectors. Another interesting fact is that one can obtain maxmin-$\omega$ eigenvectors  from some columns of the (refined) DBM $\overline{D}_p$. The following example showcases this finding.

\begin{exmp}
\normalfont
\label{exmp:interesting}
    Suppose we have a matrix
    \[
    A=\begin{bmatrix}
        4&13&5&2\\
        10&4&18&0\\
        5&14&19&7\\
        18&2&9&0
    \end{bmatrix}.
    \]
    For $\omega=1/2$, by using \Cref{a:sequence,alg:refinement}, one can obtain $\lambda_2=4$ and
    \[
    \overline{D}_2=\begin{bmatrix}
        0&-4&-5&-2\\
        -9&0&-10&2\\
        1&1&0&3\\
        -11&-2&-12&0\\
    \end{bmatrix},
    T=\begin{bmatrix}
        4&\cdot&\cdot&\cdot\\
        \cdot&4&\cdot&\cdot\\
        \cdot&\cdot&\cdot&\cdot\\
        \cdot&2&\cdot&\cdot\\
    \end{bmatrix},
    \hat{A}=\begin{bmatrix}
        4&\cdot&\cdot&\cdot\\
        \cdot&4&\cdot&\cdot\\
        5&15&\cdot&7\\
        \cdot&2&\cdot&\cdot\\
    \end{bmatrix}.
    \]
    Observe that 
    \[A\otimes_{1/2} \overline{D}_2-
    \overline{D}_2=
    \begin{bmatrix}
        4&4&4&4\\
        4&4&4&4\\
        4&4&0&4\\
        4&4&4&4
    \end{bmatrix}
    \]
    which confirms that the 1st, 2nd, and 4th columns of $\overline{D}_2$ are indeed eigenvectors associated with the eigenvalue $\lambda_2=4$. One could check that two of these columns, the 1st and 2nd ones, are essentially different. Interestingly, the 3rd column of $\overline{D}_2$ can be expressed as a min-plus combination of the first two columns of $\overline{D}_2$ namely $\overline{D}_2(\cdot,3)=(-1)\otimes \overline{D}_2(\cdot,1) \oplus^\prime (-1)\otimes \overline{D}_2(\cdot, 2)$, but it is not an eigenvector.
    
\end{exmp}

To show the performance of \Cref{alg:refinement} (and \Cref{a:sequence} as its implicit part), we present some numerical experiments. For all experiments, the elements of matrices are integers in $[0,100]$, and $100$ matrices are generated for each dimension $5\leq n\leq 10$. The experiments have been implemented in Python on an AMD Ryzen 4800H, 2.90 GHz and 16GB of RAM. For each $n$, we run the experiments for $\omega=p/n$ where $2\leq p\leq n-1$. We do not include the cases when $p=1$ and $p=n$ as they represent the min-plus and max-plus eigenproblems, respectively. Hence, there are 3300 experiments in total.

\Cref{tab:running-time} indicates the results of the experiments. Each cell shows the average running time out of the successful ones and the success rate which is shown in brackets. Although the proposed algorithms are heuristic, the success rate is quite high. Based on the table, it is also evident that the value of $p$ heavily affects the running time. Furthermore, for each dimension, the algorithm runs the slowest when $p\approx n/2$. This is due to \Cref{eq:min-max-block-v2} involving minimization with $\binom{n}{p}\times \binom{n}{n+1-p}$ terms. 

\begin{table}[!ht]
    \centering
    \caption{The average running time in second and the success rate.}
     \label{tab:running-time}
\begin{tabular}{|c|c|c|c|c|c|c|c|c|c|c|}
\hline
\diagbox[width=1cm, height=1cm]{$n$}{$p$}&2&3&4&5&6&7&8&9\\
\hline
\multirow{2}{*}{5}&0.27 &0.48 &0.21&\multirow{2}{*}{-}&\multirow{2}{*}{-}&\multirow{2}{*}{-}&\multirow{2}{*}{-}&\multirow{2}{*}{-}\\
&(100) &(100) &(99)&&&&&\\\hline
\multirow{2}{*}{6}&0.69 &2.29 &2.58&0.70&\multirow{2}{*}{-}&\multirow{2}{*}{-}&\multirow{2}{*}{-}&\multirow{2}{*}{-}\\
&(100) &(100)&(100) &(99)&&&&\\\hline
\multirow{2}{*}{7}&1.66 &8.73 &14.74&8.20&1.80&\multirow{2}{*}{-}&\multirow{2}{*}{-}&\multirow{2}{*}{-}\\
&(100) &(100) &(100)&(100)&(100)&&&\\\hline
\multirow{2}{*}{8}&4.49 &27.32 &73.72&74.53&26.71&3.68&\multirow{2}{*}{-}&\multirow{2}{*}{-}\\
&(100) &(100) &(100)&(99)&(99)&(100)&&\\\hline
\multirow{2}{*}{9}&9.97 & 93.92 &391.79 &466.37 & 338.29& 94.24&17.36 &\multirow{2}{*}{-}\\
&(100) &(99) &(100)&(100)&(100)&(99)&(100)&\\\hline
\multirow{2}{*}{10}& 15.75& 189.00 & 927.21& 2135.31&2110.00 &989.90 &197.87 &15.80\\
&(100) &(100) &(100)&(100)&(100)&(100)&(100)&(100)\\\hline
\end{tabular}
\end{table}

For all successful experiments conducted, it is always the case that there exists $i\in [n]$ such that $\overline{D}_p(\cdot,i)$ is an eigenvector, like in \Cref{exmp:interesting}. This suggests the following conjecture.
\begin{con}
    For each matrix $A\in \R^{n\times n}$ with rational entries and $\omega=p/n$ for $p\in [n]$, for which  \Cref{alg:refinement} yields the eigenvalue $\lambda_p$, it also yields an eigenvector which can be obtained from at least one column of $\overline{D}_p$.
\end{con}

\section{Conclusions}
This paper investigates eigenproblems \(A \otimes_{\omega} x = \lambda \otimes x\) within maxmin-\(\omega\) systems. These problems extend the min-plus and max-plus eigenproblems \(A \otimes x = \lambda \otimes x\) and \(A \otimes' x = \lambda \otimes x\), where a threshold \(\omega \in (0, 1]\) determines which term \(A(i, j) + x_j\) is selected in each \(i\)-th equation. We propose a novel method utilizing zone manipulation to over-approximate the eigenspace, expressed through Difference Bound Matrices (DBMs). This allows for some matrix-based operations to find the eigenspace over-approximation. We have constructed an over-approximation sequence of zones that eventually stabilizes, also enabling the determination of what we call the strongly active element matrix. If this matrix contains at least one finite entry in each row, the eigenvalue can be easily computed from the sole critical cycle in the strongly active graph. Otherwise, \Cref{alg:refinement} is used to refine the approximation. Numerical examples show that the proposed algorithms have a high success rate. Interestingly, in every successful experiment, at least one column of the DBM representation of the (refined) approximation set is always an eigenvector corresponding to the obtained eigenvalue.

For future work, we aim to investigate eigenproblems in maxmin-\(\omega\) systems using alternative methods such as the power algorithm and policy iteration.


\end{document}